\documentclass[11pt,reqno]{amsart}
\usepackage[dvipsnames]{xcolor}
\usepackage{amsmath, amscd, amsthm, amssymb, mathrsfs}

\usepackage{comment}
\usepackage[textsize=tiny]{todonotes} %julien added

\usepackage[colorlinks=true]{hyperref}  %julien added

\newtheorem{theorem}{Theorem}[section]
\newtheorem*{theorem*}{Theorem}

\newtheorem{lemma}[theorem]{Lemma}
\newtheorem{proposition}[theorem]{Proposition}
\newtheorem{innercustomthm}{Theorem}
\newenvironment{customthm}[1]
  {\renewcommand\theinnercustomthm{#1}\innercustomthm}
  {\endinnercustomthm}
  
\theoremstyle{definition}
\newtheorem{definition}[theorem]{Definition}
\newtheorem*{definition*}{Definition}
\theoremstyle{remark}
\newtheorem{remark}[theorem]{Remark}

\numberwithin{equation}{section}

\newcommand{\Tr}{\mathrm{Tr}}

\newcommand{\rk}{\mathrm{rk}}
\newcommand{\tr}{\mathrm{tr}}

\newcommand{\Ker}{\mathrm{Ker}}

\newcommand{\Id}{\mathrm{Id}}
\newcommand{\Scal}{\mathrm{scal}}
\newcommand{\End}{\mathrm{End}}
\newcommand{\Hom}{\mathrm{Hom}}
\newcommand{\Met}{\mathrm{Met}}

\newcommand{\Grass}{\operatorname{Grass}}
\newcommand{\Map}{\operatorname{Map}}

\newcommand{\CC}{{\mathbb C}}

\newcommand{\RR}{{\mathbb R}}

\newcommand{\vertiii}[1]{{\left\vert\kern-0.25ex\left\vert\kern-0.25ex\left\vert #1 
    \right\vert\kern-0.25ex\right\vert\kern-0.25ex\right\vert}}

\title[Quantization of Hitchin's equations for Higgs Bundles]{Quantization of Hitchin's equations  \\for Higgs Bundles I}

\author[M.~Garcia-Fernandez]{Mario Garcia-Fernandez}
  \address{MGF: Instituto de Ciencias Matem\'aticas, Nicol\'as Cabrera, no. 13-15, Campus de Cantoblanco, 28049 Madrid Spain}
\email{mario.garcia@icmat.es }

\author[J.~Keller]{Julien Keller}
  \address{JK: Centre de Math\'ematiques et Informatique, Aix-Marseille Universit\'e, 39, rue F. Joliot Curie, 13453 Marseille Cedex 13, France}
  \email{julien.keller@univ-amu.fr}

\author[J.~Ross]{Julius Ross}
  \address{JR: Department of Pure Mathematics and Mathematical Statistics, University of Cambridge, Wilberforce Road, Cambridge, CB3 0WB, UK}
  \email{j.ross@dpmms.cam.ac.uk}


\begin{document}
\begin{abstract}
 We provide an algebraic framework for quantization of Hermitian metrics that are solutions of the Hitchin equation for Higgs bundles over a projective manifold.
 Using Geometric Invariant Theory, we introduce a notion of balanced metrics in this context. We show that balanced metrics converge at the quantum limit towards the solution of the Hitchin equation. We relate the existence of balanced metrics to the Gieseker stability of the Higgs bundle.
\end{abstract}

\date{January 19, 2016}

 \maketitle
 
 \setcounter{tocdepth}{2}
% \tableofcontents
% \todo[inline]{Some notational choices currently in use (I am happy to discuss, but otherwise lets stick to what is here).  1. $\Lambda$ instead of $\Lambda_{\omega}$.  2. Gieseker semistable instead of Gieseker-semistable 3. $\rk(E)$ instead of $rank(E)$ 5. $V$ is the volume of $X$ (it is also used in a couple of places for a vector space which should not cause confusion, but could easily be removed) 6.  $N= H^0(E\otimes L^k)$ throughout 6. Bergman function not Bergman kernel. 7. $\mathfrak{C}_{\alpha\beta}$ instead of $\mathfrak{C}_{\alpha,\beta}$ 8. $\Id_E$ for the identity endomorphism of $E$ instead of $\Id$ 8. We use $s^{*h}$ and not $s^{*_h}$ or $s^{*,h}$ and similarly for $\phi$.  Also all metrics are taken to be on $E$, so we have to tensor by $h_L^{\pm k}$ in the relevant places; this is a bit cumbersome so we can discuss an abuse of notation if you like.}

 \section*{Introduction}
The techniques of quantization in complex geometry give a way to approximate analytically defined objects by algebraic ones.  Two beautiful examples of this are the quantization of hermitian metrics on complex vector bundles over a K\"ahler manifold, in particular Hermitian-Einstein metrics, and the quantization of K\"ahler metrics on complex manifolds, in particular extremal K\"ahler metrics.  In both cases the theory has practical consequences, such as providing a method to compute numerical approximations to these metrics, as well as theoretical ones, such as uniqueness theorems and illuminating the connection with stability.

In this paper we provide a framework for quantization of metrics on a Higgs bundle, by which we mean a pair $(E,\phi)$ consisting of a holomorphic vector bundle $E$ of rank $\rk(E)\ge 2$ over a K\"ahler manifold $(X,\omega)$ and a holomorphic map
$$\phi\colon E\to E\otimes \Omega^1_X$$
called the \emph{Higgs field} that is required to satisfy $\phi\wedge \phi=0$.  A hermitian metric $h$ on $E$ is said to satisfy the \emph{Hitchin equation} (also called in the literature Hermitian-Yang-Mills equation for Higgs bundle) if
$$ \Lambda\left( \sqrt{-1} F_{h} + [\phi, \phi^{*h}]\right) = \frac{\mu(E)}{V} \Id_E.$$
Here $F_{h}$ denotes the curvature of the  Chern connection associated to $h$, $\phi^{*h}$ denotes the image of $\phi$ by the combination of the anti-holomorphic involution in $\End(E)$ determined by $h$ with the conjugation of 1-forms, $\Lambda$ is the contraction with the K\"ahler form $\omega$, $V$ is the volume of $(X,\omega)$ and $\mu(E) = \operatorname{deg}(E)/\rk(E)$ is the slope of $E$.
%nd $f$ is a smooth function chosen so that 
%$$ \int_X f \frac{\omega^{n}}{n!}=\mu(E)=: \frac{\deg(E)}{\rk(E)}.$$
There is an analogue of the Donaldson--Uhlenbeck--Yau Theorem due to Hitchin and Simpson \cite{Hitchin,Simpson1}, which states that $E$ admits a hermitian metric that solves the Hitchin equation  if and only if $(E,\phi)$ is slope-polystable.  \medskip

In this paper we define a notion of \emph{balanced metric} for Higgs bundles with the following features:
\begin{enumerate}
\item Balanced metrics are finite dimensional approximations of solutions to the Hitchin equation, and
\item The existence of balanced metrics is related to a form of stability of the underlying Higgs bundle just as in Hitchin--Simpson Theorem.
\end{enumerate}

We define this balanced condition in two ways, first in terms of a moment map and second using the density of states function (also known as the Bergman function).  For both of these we must assume the K\"ahler form $\omega$ is integral, so lies in $c_1(L)$ for some ample line bundle $L$, and let $h_L$ be a hermitian metric on $L$ whose Chern connection has curvature $-\sqrt{-1}\omega$.    Given any basis $\underline{s}$ for $H^0(E\otimes L^k)$, the evaluation map induces a holomorphic
$$ u_{\underline{s}} \colon X\to \Grass(\mathbb C^{N}; \rk(E))=:\mathbb G$$
from $X$ to the Grassmannian of $\rk(E)$-dimensional quotients of $\mathbb C^{N}$ where $N:=N_k:= h^0(E\otimes L^k)$.   Moreover for $k$ sufficiently large $u_{\underline{s}}$ is an embedding, and $E\otimes L^k\simeq u_{\underline{s}}^* \mathcal U$ where $\mathcal U$ denotes the universal quotient bundle on $\mathbb G$.  We write the space of such maps as
$$\Map:= \Map_k := \{u: X\to \mathbb G \text{ such that $u$ is holomorphic}\}.$$
Then there is a space 

\begin{equation}\label{eq:Zk}
Z:=Z_k\stackrel{\pi}{\to} \Map
\end{equation}
whose fiber over $u\in \Map$ is
%\mario{I found the original way of writting this very confusing. I changed it.}
$$ Z|_{u} = H^0(\End(u^*\mathcal U) \otimes \Omega^1_X).$$
%$$ Z|_{u} = H^0(\End(E) \otimes \Omega^1_X) \text{ where } E = u^*\mathcal U\otimes L^{-k}.$$
We define a K\"ahler structure on the regular part of $Z$ as follows.  First there is a K\"ahler form $\Omega^{\Map}$ on $\Map$ (as used by Wang \cite{W2}) given by
$$ \Omega^{\Map}|_u(\eta_1,\eta_2) = \int_X u^*\omega_{\mathbb G}(\eta_1,\eta_2) \frac{\omega^{n}}{n!}\text{ for } \eta_1,\eta_2 \in T_u\Map = H^0(u^*T\mathbb G)$$
where $\omega_{\mathbb G}$ denotes the standard Fubini-Study form on $\mathbb G$.
Next observe that for any  $\phi\in Z|_u = H^0(\End(u^*\mathcal U)\otimes \Omega_X^1)$ the pullback of the Fubini-Study metric $h_{FS}$ on $\mathcal U$ along with the hermitian metric on $\Omega_X^1$ induced by $\omega$ gives a pointwise hermitian metric on $\phi$.   We then let
$$ \Omega_k: = \pi^* \Omega^{\Map}+ \frac{\lambda}{4} dd^c \int_X \log(1 + \lambda^{-1}k^{-1}|\cdot|^2)$$
where the $dd^c$ means taking derivatives in the directions in $Z$, and $\lambda$ is the real constant
$$\lambda : = \frac{1}{2(\rk(E)-1)}.$$
We will see that $\Omega_k$ is a positive closed $(1,1)$-form that is invariant with respect to the naturally induced $SU_N$-action, and this action is hamiltonian so admits a moment map.

\begin{definition*}
A hermitian metric $h$ on $E$ is \emph{balanced at level $k$} if is of the form
$$ h = u_{\underline{s}}^* h_{FS} \otimes h_L^{-k}$$
where $\underline{s}$ is a basis of $H^0(E\otimes L^k)$ such that the point
$$(u_{\underline{s}},\phi)\in Z$$ 
is a zero of the moment map for the $SU_N$ action with respect to $\Omega_k$.
\end{definition*}

In fact in Section \ref{sec:balanceddefinition}  the form on $Z$ we consider will depend on two real parameters $\alpha,\beta$, and what is described in this introduction is a special case.  The precise choice of the constant $\lambda$ is unimportant; for the positivity of $\Omega_k$ we may take any constant smaller than $2(\rk(E)-1)^{-1}$.

As we will show, this definition can be recast intrinsically in terms of the Bergman function.   Given any hermitian metric $h$ on $E$, we have an $L^2$-inner product on $H^0(E\otimes L^k)$ induced by $h$, $h_L$ and the volume form determined by $\omega$.  The Bergman function is then defined to be
$$ B_k(h) : = \sum_i s_i\otimes s_i^{*{h\otimes h_L^k}}\in C^{\infty}(\End(E))$$
where $\{s_i\}$ is any orthonormal basis for $H^0(E\otimes L^k)$.  We let
$$\mathfrak{C}_k(h) :=  \frac{1}{k(1+ \lambda^{-1}k^{-1} |\phi|_h^2)} \Lambda [\phi,\phi^{*h}]$$
which is a smooth endomorphism of $E$ that is hermitian with respect to $h$.

\begin{definition*}
A hermitian metric $h$ on a Higgs bundle $(E,\phi)$ is \emph{balanced} at level $k\in \mathbb N$ if  the hermitian metric
$$ \widehat{h} : = h ( \Id_E - \mathfrak{C}_k(h))$$
satisfies
$$ B_k( \widehat{h}) = \frac{N}{\rk(E) V} (\Id_E- \mathfrak{C}_k(h))$$
where $N:=h^0(E\otimes L^k)$ and $V := \int_X \frac{\omega^n}{n!}$. 
\end{definition*}

Equivalently $h\in \Met(E)$ is balanced at level $k$ if $$\sum_i t_i \otimes t_i^{*h\otimes h_L^k} = \frac{N}{\rk(E) V} \Id_E$$
where $\{t_i\}$ is a basis for $H^0(E\otimes L^k)$ that is orthonormal with respect to the $L^2$-inner product defined using $\widehat{h}$ (see Lemma \ref{lem:reformulationbalanced}).
%\julius{Jan 15: added equivalent def of balanced to the intro as it is in some ways clearer}

%The connection with the Hitchin Equation now follows from the asymptotics of the Bergman function (see Section \ref{}), since if $h\in \Met(E)$ is given, the right hand side of \eqref{eq:balancedbergman} has an asymptotic expansion
%$$ \Id_E + \frac{1}{k} \left( \Lambda \left(F_h + [\phi,\phi^{*h}]\right) + \frac{1}{2}\Scal(\omega)\Id \right)+ %O\left(\frac{1}{k^2}\right).$$ 

\subsection*{Statement of Results}

%\julius{Dec 15th paragraph changed reflecting the split of this paper into two pieces} 
The purpose of this paper is to describe what we believe to be the correct framework for quantization of the Hitchin equation.  We prove two theorems that support this, with the expectation of providing more in the future.  The first gives connection between balanced metrics and the Hitchin equation asymptotically as $k$ tends to infinity, and to state it precisely let $\Scal(\omega)$ denote the scalar curvature of $\omega$, and $\Scal(\omega)_0 := \Scal(\omega) - \overline{S}$ where $\overline{S}$ is the average of $\Scal(\omega)$ over $X$.  We let $\Met(E)$ denote the set of hermitian metrics on a Higgs bundle $(E,\phi)$.

\begin{customthm}{A}(Theorem \ref{thm1})
Suppose $h_k\in \Met(E)$ is a sequence such that $h_k$ is balanced at level $k$ that converge to $h_{\infty}\in \Met(E)$ as $k$ tends to infinity.  Then $h_{\infty}$ satisfies the equation\begin{equation}
 \Lambda \left(  \sqrt{-1} F_{h_\infty} + [\phi, \phi^{*h_\infty}]\right) = \left(\frac{\mu(E)}{V} - \frac{1}{2}\Scal_0(\omega)\right)\Id_E. \label{eq:f}
 \end{equation}
 Thus after a possible conformal change, $h_{\infty}$ satisfies the Hitchin equation.
\end{customthm}

%The first two theorems generalise what is already known in the absolute case    Roughly speaking they say that balanced metrics are good approximations of the Hitchin equation.

In the absolute case (i.e. without a Higgs field)  there is a converse to this statement proved by Wang \cite{W2} (following ideas of Donaldson \cite{D1}).    In a sequel to this paper we will provide the analogous statement for Higgs bundles and prove that the existence of a solution to the Hitchin equation implies the existence of balanced metrics for $k$ sufficiently large.

%\begin{customthm}{B}(Theorem \ref{thm2})
%Suppose $E$ is irreducible and that $h_{\infty}\in\Met(E)$ satisfies the Hitchin equation \eqref{eq:f}.  Then there exists a sequence $h_k\in\Met(E)$ converging to $h_{\infty}$ as $k$ tends to infinity, such that $h_k$ is balanced at level $k$.
%\end{customthm}

On the algebraic side we give the following Hitchin--Simpson type statement for Higgs bundles that relates the existence of balanced metrics to a form of stability (c.f.\ Wang \cite{W2} in the absolute case):

\begin{customthm}{B}(Theorem \ref{thm3})
There exists a $k_0$ such that for all $k\ge k_0$ the following holds:  if $(E,\phi)$ admits a hermitian metric that is balanced at level $k$ then it is Gieseker semistable.  If moreover $E$ is irreducible then $(E,\phi)$ is Gieseker stable.
\end{customthm}

\subsection*{Context and Comparison with other work}
The moduli space of Higgs bundles can be thought about in various ways.  Analytically, letting $\mathcal J$ denote the space of holomorphic structures on a complex vector bundle $E$, the moduli space is obtained as the quotient of the set of those $(J,h,\phi)$ in $\mathcal Z:=\mathcal J\times \Met(E)\times \Omega^{1,0}(X,\End(E))$ such that $h$ is compatible with $J$ and
$$ \Lambda (  \sqrt{-1} F_{h,J} + [\phi,\phi^{*h}]) =\operatorname{const} \Id\text{ and } \overline{\partial}_J \phi=0$$
by the group $\mathcal{G}=GL(E)$ of gauge transformations; so as a quotient of an infinite dimensional space by an infinite dimensional group. On the other hand, through the Hitchin--Simpson Theorem, it can also be described algebraically as a space of (poly)stable objects, so as a Geometric Invariant Theory quotient of a finite dimensional space $Z$ by a finite dimensional group $G$.  There are different ways in which this can be done, and here we identify one that reflects many aspects of the infinite dimensional picture. 

Slightly more precisely we may think of the quotient of the space $Z_k$ described in \eqref{eq:Zk} by the action of $GL_{N_k}$ as a moduli space of Higgs bundles (of some kind).   Moreover these spaces $Z_k$ carry with them various structures that approximate the infinite dimensional structures as $k$ tends to infinity, roughly summarized by the following dictionary:

%{\vspace{2mm}
{\small 
\begin{center} 
\renewcommand{\arraystretch}{1.1}
 \begin{tabular}{ |c | c| } 
\hline 
\textbf{Algebraic} & \textbf{Analytic} \\
& \\
\hline 
 $Z_k$ & $\mathcal Z$ \\
 \hline
 $GL_{N_k}$ & $\mathcal G$ \\ \hline
% $Z_k/GL_{N_k}$ & $\mathcal Z/\mathcal G$\\
 GIT stability & Slope stability\\ \hline
 Balanced equation & Hitchin equation\\ \hline
 Negative Gradient flow & Donaldson heat flow\\ 
 of the balancing map & \\ \hline
 Iterative method & Discretization \\
                  & of Donaldson heat flow \\
%\\ \hline $\Omega_k$ & $L^2$-metric\\
 \vdots&\vdots\\
  \hline
 \end{tabular}
 \end{center}\vspace{2mm}}
It is interesting to ask if other aspects of the moduli space of Higgs bundles are also reflected algebraically in this way (for instance the hyperk\"ahler structure, the integrable system, or compactification considerations) but we leave such questions for future consideration.

A related notion of balanced metric was introduced by JK in \cite{Keller}, for suitable quiver sheaves arising from dimensional reduction considered in \cite{AC3}, but as pointed out in \cite{AC2} this does not allow twisting in the endomorphism and thus does not apply to  Higgs bundles. We remark also that our definition differs from that of L. Wang for which a link with stability was missing \cite[Remark p.31]{WangLi}.  In previous work of MGF and JR \cite{GFR} a different parameter space was used giving a slightly different balanced condition (that considers the case of twisted Higgs bundles with globally generated twist).   This different parameter space has the advantage in that the link with stability is easier to see, but the disadvantage that the link with the Hitchin equation is harder. 

As we will see, asymptotically as $k$ tends to infinity the condition that $h\in \Met(E)$ be balanced is that
\begin{equation}\label{eq:balancedintro4again}
B_k(h) + k^{n-1} [\phi,\phi^{*h}] = \Id + O(k^{n-2})
\end{equation}
This was the equation considered by Donagi--Wijnholt \cite[\S 3.3]{DonWi} which in fact was the original motivation of the authors for looking at balanced metrics on Higgs bundles.  What we define in this paper is a refinement of \eqref{eq:balancedintro4again} that fits into a moment map framework. In fact,  Donaji--Wijnholt consider this so as to discuss an iterative method, which can be used to numerically calculate approximate solutions to the Hitchin equation.  In the absolute case (i.e.\ without the Higgs field) this iteration has been carried out \cite{Douglas}.

The results in this work generalize to arbitrary twisted quiver bundles with relations, as studied in \cite{AC2,AC3,Schmitt1,Schmitt2,Schmitt3}. In this more general set-up, the existence of solutions
of the \emph{twisted quiver vortex equations} is related with the slope stability of the twisted
quiver bundle. A notion of Gieseker stability for twisted quiver sheaves has been
provided in \cite{AC1,Schmitt2} for the construction of a moduli space.
%Finally we remark that Higgs bundles can generalised to more ``decorated bundles" by which we mean a holomorphic vector bundle with some additional structure, and there are various works on these both algebraically (e.g.\ \cite{AC1, Schmitt1,Schmitt2,Schmitt3}) and analytically (e.g.\ \cite{AC2,AC3}). 
We have focused our attention on Higgs bundles, %since it seems to be the most important, 
but it is not hard to see that the parameter space we use here generalises to cover also %many of 
these cases, and there is %should be 
an analogous definition of balanced metric.   There is also a definition of Higgs principle bundle \cite{Gomez1,Gomez2} and, more recently, generalized quiver bundle \cite{Araujo}, and it may be interesting to know what the definition of balanced metric is in this case.

\subsection*{Acknowledgments}
JR is supported by an EPSRC Career Acceleration Fellowship (EP/J002062/1).
As Bye-fellow of Churchill college, JK is very grateful to Churchill college and Cambridge University for providing him excellent conditions of work during his stay. The work of JK has been carried out in the framework of the Labex Archim\`ede (ANR-11-LABX-0033) and of the A*MIDEX project (ANR-11-IDEX-0001-02), funded by the ``Investissements d'Avenir" French Government programme managed by the French National Research Agency (ANR). JK is also partially supported by supported by the ANR project EMARKS, decision No ANR-14-CE25-0010. MGF has been partially supported by the Nigel Hitchin Laboratory under the ICMAT Severo Ochoa grant No. SEV-2011-0087 and under grant No. MTM2013-43963-P by the Spanish MINECO. In the final stage of this research, MGF has been supported by a Marie Sklodowska-Curie Individual Fellowship funded by the Horizon 2020 Programme of the European Commission (655162-H2020-MSCA-IF-2014).

\section{Notation and Preliminaries}

\subsection{Preliminaries on Higgs Bundles}

Throughout $X$ will be a projective manifold of dimension $n$,  and $\omega$ a K\"ahler form on $X$ which induces a volume form $\omega^{[n]}=\frac{\omega^n}{n!}$. We let $V: = \int_X \omega^{[n]}$ denote the volume of $X$ and $\Lambda\colon \Omega^{p,q}\to \Omega^{p-1.q-1}$ denote the contraction with respect to $\omega$.  

Given a holomorphic vector bundle $E$ we let $\Met(E)$ denote the set of hermitian metrics on $E$.  Any $h\in \Met(E)$ has an associated Chern connection, whose curvature we denote by $F_h$.  For the most part we assume that $\omega$ is integral so lies in $c_1(L)$ for some ample line bundle $L$.  Then there exists an (essentially unique) $h_L\in\Met(L)$ such that $\sqrt{-1}F_{h_L} = \omega$.  %Thus any $h\in \Met(E)$ induces metrics $h\otimes h_L^k$ on $E\otimes L^k$ for every $k\in \mathbb N$, and when it cannot cause confusion we will denote this latter metric also by $h$.   \julius{added this confusion between $h$ and $h\otimes h_L^k$ as using the latter always is too cumbersome}

Given $h\in \Met(E)$ we have for each $k$ an $L^2$-inner product on $H^0(E\otimes L^k)$ induced by the hermitian metric $h\otimes h_L^k$ and the volume form $\omega^{[n]}$.   Since only the dependence on $h$ is of interest, we shall denote this by $L^2(h)$.

A \emph{Higgs bundle} consists of a pair $(E,\phi)$ where $E$ is a holomorphic bundle on $X$ and $\phi\in H^0(\End(E)\otimes \Omega^1_X)$ such that $\phi\wedge\phi=0$.     
Notice that $\omega$ induces a hermitian metric on $\Omega^1_X$, but since this is fixed we omit it from our notation.  So given $h\in \Met(E)$ we have a pointwise norm of $\phi$ that we denote simply $|\phi|_h$.  The `pointwise dual' $\phi^{*h}$ is defined by taking the image 
%hermitian dual of the endomorphism part 
of $\phi$ by the combination of the anti-holomorphic involution on $\End(E)$ induced by $h$ and complex conjugation on $1$-forms (which we shall also denote by $\phi^{*}$ when $h$ is clear from context) and is a smooth section of $\End(E)\otimes \Omega^1_X$.  More explicitly, $\phi^{*}$ is defined by the identity on 1-forms
$$
(\phi v,w)_h = (v, \phi^* w)_h,
$$
for any $v,w \in E$. Given $\phi_1,\phi_2$ smooth sections of $\End(E)\otimes \Omega_X^1$ we have the  commutator $[\phi_1,\phi_2] = \phi_1 \phi_2 + \phi_2\phi_1$ where the product means composition of the endomorphism part and wedge product of the form part, which is thus a smooth section of $\End(E)\otimes \Omega_X^2$, and is antisymmetric $[\phi_1,\phi_2] = - [\phi_2,\phi_1]$.   Finally the slope of $E$ is $\mu(E) = \deg(E)/\rk(E)$ where $\deg(E)$ is the degree of $E$ taken with respect to the class of $\omega$.  
%\julius{Dec 15: added definition of commutator}

\begin{definition}
A \emph{Higgs bundle}  $(E,\phi)$ is \emph{Gieseker stable} (resp. semistable) if for all coherent subsheaves $F\subset E$ such that $\phi(F)\subset F\otimes \Omega_X^1$ we have
$$\frac{\chi({F}\otimes L^k)}{\rk({F})}< \frac{\chi(E\otimes L^k)}{\rk(E)} \,\,\,(\text{resp. }\leq) \text{ for all } k\gg 0.$$
\end{definition}

\begin{definition}
A Higgs bundle  $(E,\phi)$ is \emph{slope stable} (resp. semistable)
if for all coherent subsheaves $F\subset E$ such that $\phi(F)\subset F\otimes \Omega_X^1$
we have
$$\mu({F})< \mu(E)  \,\,\,(\text{resp. }\leq) $$
We say that the Higgs bundle is slope polystable if it is a direct sum of Higgs bundles of the same slope.
\end{definition}

This notion of Gieseker stability appeared in the work of Simpson \cite{SimpsonII}.  As in the usual case, slope stability of a Higgs bundle implies Gieseker stability, and  Gieseker semistability implies slope semistability, and a slope stable Higgs bundle is simple, i.e has no non-trivial holomorphic endomorphism.

%\subsection{The Hitchin-Kobayashi Correspondence}
\begin{theorem}[Hitchin--Simpson Theorem]
Let $(E,\phi)$ be a Higgs bundle.
 The following statements are equivalent:
 \begin{enumerate}
  \item The Higgs bundle $(E,\phi)$ is slope polystable. 
%  \item There exists an $h\in \Met(E)$ that solves the equation 
% \begin{equation}\label{Higg-standard}\Lambda (  \sqrt{-1} F_{h}+[\phi,\phi^{*{h}}])=\frac{\mu(E)}{V}\Id_E.\end{equation}
 % \item For any (resp.\ for all) $f\in C^{\infty}_{\mu(E)}(X)$, there exists an $h\in \Met(E)$ that solves the equation 
 % \begin{equation}\label{Higgs-2}\Lambda (  \sqrt{-1} F_{h}+[\phi,\phi^{*{h}}])=f \Id_E.
 % \end{equation}
    \item For any $\tau\in \mathbb R_{>0}$ there exists an $h\in \Met(E)$ that solves the equation 
    \begin{equation}
    \Lambda(  \sqrt{-1} F_{h}+\tau[\phi,\phi^{*{h}}])= \frac{\mu(E)}{V} \Id_E.\label{Higgs-3}
    \end{equation} 
     % \item Given a function $f$ with $f-\mu(E)\in C^{\infty}_0(M,\mathbb{R})$, there exists $\epsilon>0$ and a metric $h\in Met(E)$ solution of the equation 
%  \begin{equation}\label{Higgs-ours}\Lambda \left(F_{h}+\frac{1}{1+\epsilon\vert\phi\vert^2_h}[\phi,\phi^{*_{h}}]\right)=f Id_E
 \end{enumerate}
\end{theorem}

\begin{remark}
This theorem is due to Hitchin \cite{Hitchin} when $\dim X=1$ and Simpson \cite{Simpson} in all dimensions.  Although this is often stated as an equivalence when $\tau=1$ we remark that the above statement follows  immediately since $(E,\phi)$ is slope-polystable if and only if $(E,\tau \phi)$ is (although there is usually no easy way to pass between the solutions that solves \eqref{Higgs-3} for different values of $\tau$).  We remark also that this correspondence does not require $\omega$ to be integral.
\end{remark}

\subsection{Fubini-Study metric on the Grassmannians}

We collect some standard facts about Fubini-Study metrics on the Grassmannian.  Let $\mathbb G:=\mathbb G(N-r,\mathbb C^N)$ denote the Grassmannian of $(N-r)$-dimensional subspaces of $\mathbb C^N$.    Denote the tautological bundle by
$$ \mathcal S = \{ (\Lambda, v)\in \mathbb G\times \mathbb C^N  : \Lambda\in \mathbb G, v\in \Lambda\}.$$
We want also to consider $\mathbb G$ as the space of $r$-dimensional quotients of $\mathbb C^N$, which has a universal quotient bundle sitting in the exact sequence $0\to \mathcal S\to \mathbb C^N\otimes \mathcal O_{\mathbb G}\to \mathcal U\to 0$.   We think of this exact sequence as being a sequence of hermitian bundles, whose middle term $\mathbb C^N\otimes \mathcal O_{\mathbb G}$ is the trivial bundle with the constant metric given by the standard metric on the $\mathbb C^N$ fibres (and thus is flat when thought of as a hermitian metric on a bundle over $\mathbb G$).  This induces hermitian metrics $H_{\mathcal S}$ and $H_{\mathcal U}$ on $\mathcal S$ and $\mathcal U$ respectively, and we shall call the metric $H_{\mathcal U}$ the \emph{Fubini-Study} metric and denote it also by $h_{FS}$.   We also define a K\"ahler-form $\omega_{\mathbb G}$
by
$$ \omega_{\mathbb G}  =  \sqrt{-1} F_{\det(H_\mathcal S^*)} = -\Tr  \sqrt{-1} F_{H_\mathcal S},$$
which we refer to as the \emph{Fubini-Study form}.    One computes easily using the second-fundamental form $\beta\in \mathcal A^{1,0}(\Hom(\mathcal S,\mathcal U))$ \cite[p.78]{Griffiths} that
\begin{equation}
 \tr \sqrt{-1}F_{H_\mathcal U} =   \sqrt{-1} \tr (\beta \wedge \beta^*)= -  \sqrt{-1} \tr (\beta^*\wedge \beta) =  -\tr   \sqrt{-1}F_{H_\mathcal S}  = \omega_{\mathbb G}.\label{eq:curvatureuniversal}
\end{equation}

\section{Balanced Metrics: definition and reformulation}\label{sec:balanceddefinition}

\subsection{Parameter Spaces}

We return now to the moment map framework for the balanced condition, for which we require a finite dimensional parameter space.    First we recall the analogous picture for the quantization of metrics on holomorphic bundles, as considered by Wang \cite{W2}.  As above $X$ is to be a projective manifold of dimension $n$ and $L$ an ample line bundle on $X$.    Fix a polynomial $N=N(k)= rk^n + \cdots$ where $r\in \mathbb N$ (which should be thought of as the Hilbert polynomial of the vector bundles we will eventually consider).   Let
$$\mathbb G =  \Grass(\mathbb C^{N}; r)$$
be the Grassmannian of $r$-dimensional quotients which carries a tautological quotient bundle
$$\mathbb C^{N}\otimes \mathcal O_\mathbb G \to \mathcal U\to 0$$
We then let
$$\Map := \{ u : X\to \mathbb G : u \text{ is holomorphic}\}$$
whose tangent space at $u\in \Map$ is
$$T_u\Map = H^0(u^* T\mathbb G).$$
To incorporate the Higgs field, consider the universal evaluation map
$$e:\Map\times X \to \mathbb G \text{ given by } e(u,x) = u(x)$$
and define
$$\mathcal Z : = (\pi_1)_* (e^* \End(\mathcal U) \otimes \pi_2^* \Omega^1_X),$$
where $\pi_1\colon \Map\times X\to \Map$ and $\pi_2\colon \Map\times X\to X$ are the projections.  Thus the stalk of $\mathcal Z$ over $u\in \Map_k$ is precisely
$$\mathcal Z|_u = H^0(\End(E)\otimes \Omega^1_X) \text{ where } E:= u^*\mathcal U\otimes L^{-k}.$$
\begin{definition}
We let $Z$ be the total space (i.e. the space of stalks) of $\mathcal Z$ and $\pi\colon Z\to \Map$ be the natural projection.
\end{definition}

Of course this whole construction depends on the parameter $k$ that has been omitted from notation.  The way in which $Z$ parameterizes Higgs bundles $(E,\phi)$ is obvious, for if $E$ has rank $r$ and Hilbert polynomial $N$ then for $k$ sufficiently large $h^0(E\otimes L^k) = N(k)$,  and any basis $\underline{s}$ for $H^0(E\otimes L^k)$ induces a $u_{\underline{s}}\in \Map$. Moreover, $u_{\underline{s}}$ is even an embedding for $k$ sufficiently large,  $u_{\underline{s}}^* \mathcal U = E\otimes L^k$, and the Higgs field gives, for each choice of $\underline{s}$, a point in $Z|_{u_{\underline{s}}} = H^0(\End(E)\otimes \Omega_X^1)$.

% The group $GL_{N(k)}$ acts naturally on $Z$, and the orbits correspond to isomorphism classes of Higgs bundles.

\subsection{K\"ahler structure on $Z$}

 Observe that $Z$ need not be smooth, even over the smooth locus of $\Map$,  since $\mathcal Z$ may not be locally free.  One way to deal with this is to replace $Z$ with its smooth locus, or otherwise replace $Z$ with it restriction to a subset of the regular locus $\Map$ over which $\mathcal Z$ is locally free.  We presume that one of this options has been chosen, and for simplicity denote the new space also  by $Z$ (we will later be interested only in a single $GL_N$ orbit of $Z$ so this replacement has no effect).

As discussed in the introduction, the smooth locus of $\Map$ has a K\"ahler structure given by
$$ \Omega^{\Map}_u(v,v') : = \int_X \omega_{\mathbb G}(v,v') \omega^{[n]} \text{ for } v,v'\in T_u\Map = H^0(u^*T\mathbb G),$$
where $\omega_{\mathbb G}$ denotes the Fubini-Study form of the Grassmannian $\mathbb G$.   It is immediate (see \cite{W2}) that $\Omega^{\Map}$ is K\"ahler.

 For the fiber directions in $Z$ we construct a potential as follows.  Let $h_{FS}$ denote the Fubini-Study hermitian metric on $\mathcal U$.  Then given $u\in \Map$ there is an induced hermitian metric on $\End(u^* \mathcal U)\otimes \Omega_X^1$ obtained by the tensor product of the pullback of $h_{FS}$ and the hermitian metric on $\Omega_X^1$ induced by $\omega$.  By abuse of notation we denote this simply by $u^*h_{FS}$.  

\begin{definition}
Fix $\alpha>0$.  Then for $u\in \Map$ and $\phi\in Z|_u = H^0(\End(u^*\mathcal U)\otimes \Omega_X^1)$ let
$$ \Gamma (\phi) := \Gamma(\phi, \alpha):= \int_X \log (1 + \alpha |\phi|_{u^*h_{FS}}^2) \omega^{[n]}.$$
\end{definition}

\begin{definition}
For $\alpha,\beta>0$ let
$$ \Omega_{\alpha,\beta}: = \pi^* \Omega^{Map} + \frac{\beta}{4} dd^c \Gamma(\cdot, \alpha),$$
where $dd^c$ denotes differentiation in the directions in $Z$.
\end{definition}

Clearly $\Omega_{\alpha,\beta}$ is a closed $(1,1)$-form, we check now that it is positive. %The action of the general linear group $GL(N)$ (as well as the unitary group $U(N)$ on $\mathbb G$, and thus induce actions also on $M$ and $Z$, and it is clear that $\Omega_{\alpha,\beta}$ is $U(N)$-invariant. 

\begin{proposition}\label{prop:kahlerform} Assume $\rk(E)>1$. 
For $\alpha>0$ and $\beta< \frac{2}{(\rk(E)-1)}$, the form $\Omega_{\alpha,\beta}$ is positive.
\end{proposition}
Before the proof we introduce some notation that we will also require later.  Fix $x\in X$ and consider the vector bundle $V=V_x$ over $\Map$ whose fibre over $u\in \Map$ is $u(x)^*\mathcal U$, so 
\begin{equation}
V= V_x = e_x^* \mathcal U\text{ where } e_x\colon \Map\to \mathbb G \text{ is }e_x(u) = u(x).\label{def:bundleV}
\end{equation}
  Then $V$ carries with it the pullback of the Fubini-Study metric given by $H_V:= e_x^* h_{FS}$ and $F_{H_V}=e_x^* F_{h_{FS}}$.  
  
Now let $\zeta$ be
%\mario{MGF: $\zeta$ is used for three different things in the next three pages. I suggest to leave this one, but denote by $y$ elements in $T_wW$}
the $n$-dimensional hermitian vector space $\Omega^1_X|_x$ along with the hermitian metric $h_\zeta$ induced by $\omega$ and define a vector bundle
\begin{equation}
W=W_x : = \End(V) \otimes \zeta = V\otimes V^* \otimes \zeta\stackrel{\pi}{\to} \Map \label{def:Wx}
\end{equation}
with the hermitian metric $H_W: = H_V\otimes H_{V^*} \otimes h_\zeta$.  Define $\Gamma_x\colon W\to \mathbb R$ by 
\begin{equation}
\Gamma_x(w) = \log ( 1 + \alpha |w|^2) \text{ for } w\in W.\label{def:Gammax}
\end{equation}

\begin{proof}[Proof of \ref{prop:kahlerform}]
Set $r:=\rk(E)$.  We claim that as long as $\beta<\frac{2}{(r-1)}$ we have
\begin{equation} 
 \pi^* e_x^* \omega_\mathbb G + \frac{\beta}{4} dd^c\Gamma_x  >0\label{eq:pointwise}
\end{equation}
as forms on $Z$ (we recall for the reader's convenience that for a $(1,1)$-form $\alpha$ to be positive means that $-\sqrt{-1}\alpha(\rho,\overline{\rho})>0$ for all $(1,0)$-vector field $\rho\neq 0$ in $W$). % \mario{MGF: I guess you mean that. Then I agree.}).

To prove \eqref{eq:pointwise}, the Chern connection of $H_W$ splits the tangent space of $W$ into vertical and horizontal parts, that we denote by $\rho_v$ and $\rho_h$ respectively for $\rho \in TW$.  Clearly $dd^c\Gamma_x$ is strictly positive in the vertical direction, and one can compute that the cross-terms vanish \cite[Lemma 2.6]{Rezza}.  Thus we are left looking at the horizontal component for which there is the following formula \cite[2.7,2.8]{Rezza}
\begin{equation}
 (dd^c \Gamma_x)_h|_w =\frac{\alpha}{1+\alpha |w|^2}\pi^* (\sqrt{-1}F_W w,w)_{H_W} \text{ for } w\in W,\label{rezzaformula}
\end{equation}
where here $dd^c$ denotes differentiation in the total space of $W$.  
Now
$$ \sqrt{-1}F_W w = [\sqrt{-1}F_V,w],$$
where $V =  e_x^* \mathcal U$ is as in \eqref{def:bundleV}, and the bracket on the right is the commutator acting only on the $\End(V)$ part of $W$.    Suppose now $\rho= \rho_h  + \rho_v \in T_wW$, so $\rho_h\in T_{\pi(w)}\Map$ and set
$$\Theta = F_V(\rho_h,\overline{\rho_h}).$$
  Then \eqref{rezzaformula} becomes
$$ -  \sqrt{-1} (dd^c\Gamma_x)_h|_w(\rho,\overline{\rho}) =  \frac{\alpha}{1 + \alpha |w|^2} ([\Theta,w],w)_{H_W},$$
and we wish to bound this quantity from below.  To this end observe that
$$([\Theta,w],w)_{H_W} = \tr(w^*(\Theta w - w \Theta)) = \tr(\Theta[w,w^*]).$$
Then as $\mathcal U$ is the quotient of a trivial bundle we have $\Theta$ is semipositive (that is, $(\Theta(v),v)_{H_V}\ge 0$ for all $v\in V$) \cite[4.3.19]{Huybrechts}.  Recall that if $A,B$ are hermitian semipositive matrices then $\tr(AB)\le \tr A\, \tr B$.  Using
\begin{equation}\label{eq:ineqbracket}
|[w,w^*]| \leq 2 |w|^2,
\end{equation}
it follows that $2|w|^2 \Id_V - [w,w^*]\ge 0$, and we have
$$ \tr \left(\Theta (2|w|^2 \Id_V - [w,w^*])\right) \le \tr \Theta \, \tr (2|w|^2 \Id - [w,w^*]) = 2r |w|^2 \tr \Theta,$$
which rearranging becomes
$$ \tr(\Theta [w,w^*]) \ge 2(1-r) |w|^2 \tr \Theta.$$
Now from the definition of the Fubini-Study metric $\omega_{\mathbb G}$ on the Grassmannian \eqref{eq:curvatureuniversal}, 
$$ \tr\sqrt{-1}\Theta = \tr e_x^* \sqrt{-1} F_{h_{FS}} (\rho_h,\overline{\rho_h}) = e_x^* \omega_{\mathbb G}(\rho_h,\overline{\rho_h})$$
so we end up with  %\mario{MGF: this was an inequality between purely imaginary numbers}
$$ - \sqrt{-1} (dd^c\Gamma_x)_h(\rho,\overline{\rho}) \ge 2(1-r)   \frac{\alpha |w|^2\tr \Theta}{1 + \alpha |w|^2} \ge 2 (1-r) (- \sqrt{-1}\pi^* e_x^* \omega_{\mathbb G}(\rho,\overline{\rho})). $$
Therefore as long as $\beta<\frac{2}{(r-1)}$ we have
\begin{equation*} 
 \pi^* e_x^* \omega_\mathbb G + \frac{\beta}{4} dd^c\Gamma_x  >0,
\end{equation*}
which proves the claim \eqref{eq:pointwise}.

The Proposition follows from this by integrating over $X$.    To see this, consider 
$$ \epsilon_x\colon Z\to W_x \text{ given by } \epsilon_x(\phi) = \phi(x) \text{ for } \phi\in Z$$
which  commutes with projection (i.e.\ covers the identity on $\Map$).   One checks easily that if $\eta \in TZ$ then
$$ (d \Gamma, \eta)= \int_X (d\Gamma_x, D\epsilon_x(\eta)) \omega^{[n]}$$,
where, we emphasize once again, the $d$ on the right hand side is taken in the total space of $W_x$.  A similar expression holds for $dd^c$, so multiplying \eqref{eq:pointwise} by $\omega^{[n]}$ and integrating over $X$ gives
$$\pi^*\Omega^{\Map} + \frac{\beta}{4} dd^c \Gamma >0$$
as by definition $\Omega^{\Map} = \int_X  (e_x^* \omega_{\mathbb G}) \omega^{[n]}$.
%\begin{equation}
%dd^c \Gamma = \int_X dd^c \Gamma_x \omega^{[n]}
%\end{equation}
\end{proof}

\subsection{Hamiltonian actions on the total space of a vector bundle}\label{sec:mmapZ}

%We turn to the moment map of this form. 
The general linear group $GL_N$ (resp.\ unitary group $U_N$) acts on $\mathbb G$, and hence induces an action on $Z$ covering an action on $\Map$. Clearly $\Omega_{\alpha,\beta}$ is $U_N$ invariant and, by construction, this action is hamiltonian. To see this, note that the $U_N$-action on $\Map$ is hamiltonian \cite{W2} and that $\Omega_{\alpha,\beta}$ is obtained from $\pi^* \Omega^{\Map}$ by adding a $U_N$-invariant exact form. To calculate the moment map, we need some generalities about hamiltonian actions on the total space of a vector bundle, that we address now. 

Let $W$ be the total space of a holomorphic vector bundle over a compact complex manifold
$$
W \to Q.
$$
Let $h_W$ be a Hermitian metric on $W$ and consider for $\alpha\in \mathbb R_{>0}$ the potential
$$
\gamma(w) =  \log (1+ \alpha|w|_{h_W}^2) \text{ for } w\in W.
$$
We define the $1$-form on $W$
$$
\sigma_W = d^c \gamma.
$$
%where $d^c =i(\overline{\partial} - \partial)$. %The following formula follows from a straightforward calculation and is left to the reader.

\begin{lemma}\label{lem:sigmaW}
Let $1_W$ be the canonical vertical vector field on $W$. Then
$$
\sigma_W(v) = \frac{2 \sqrt{-1} \alpha }{1+ \alpha|w|_{h_W}^2}\operatorname{Im} (A v,1_W)_{h_W},
$$
where $h_W$ is identified with a hermitian metric on the vertical bundle of $W$ and $A$ denotes the vertical projection with respect to the Chern connection of $h_W$.
\end{lemma}

\begin{proof}
Consider holomorphic coordinates $z_j$ on $Q$ and a holomorphic trivialization of $W$, defined on the same open patch. Given $w \in W$ in this open patch, we can identify it with $w = (z,e)$ in $U \times \CC^r$ and further we can assume that $h_W = \Id_W + O(|z|^2)$. We now have
\begin{align*}
(d^c \gamma)_{|z=0} &= \frac{ \sqrt{-1} \alpha }{1+ \alpha|w|_{h_W}^2} (\overline{\partial} - \partial)(e^* h_W e))_{|z=0},\\
& = \frac{ \sqrt{-1} \alpha }{1+ \alpha e^*e}((\partial e)^* e - e^*\partial e)),\\
& = \frac{2  \sqrt{-1}\alpha }{1+ \alpha e^*e} \operatorname{Im} (\partial e)^* e.
\end{align*}
Now, for a vector field $v$ on $W$, we can express locally $v = (\dot z, \dot e)$ and we have
$$
(A v)_{|z = 0} = (0, \dot e - \dot u\lrcorner (h_W^{-1}h_W))_{|z = 0} = (0,\dot e)_{|z = 0},
$$
which implies
$$
v \lrcorner (d^c \gamma)_{|z=0} = \frac{2  \sqrt{-1}\alpha }{1+ \alpha e^*e} \operatorname{Im} (\dot e)^* e = \Bigg{(} \frac{2 \sqrt{-1} \alpha}{1+ \alpha|w|_{h_W}^2}\operatorname{Im}(Av,1_W)_{h_W} \Bigg{)}_{|z=0}.
$$
\end{proof}

Suppose now that $U_N$ acts on $Q$ making $(W,h_W)$ a $U_N$-equivariant Hermitian holomorphic vector bundle over $Q$. 

\begin{lemma}\label{lem:mmap0}
The $U_N$-action on $W$ preserves $\sigma_W$. Moreover the action is Hamiltonian with respect to $d \sigma_W$ with moment map
$$
\langle m_W,\xi\rangle = - Y_\xi \lrcorner \sigma_W = \frac{2  \sqrt{-1} \alpha }{1+ \alpha|w|_{h_W}^2}\operatorname{Im} (A Y_\xi ,1_W)_{h_W},
$$
where $Y_\xi$ denotes the infinitesimal action of $\xi \in \mathfrak{u}_N$ on $W$.
\end{lemma}

The first equality in the previous Lemma is standard \cite[5.13]{McDuff}, while the second follows from Lemma  \ref{lem:sigmaW}. Note that the $2$-form $d \sigma_W$ may be degenerate, hence by a moment map we mean a $U_N$-equivariant smooth map $m_\alpha \colon W \to \mathfrak{u}_N^*$ satisfying the usual identity
$$
\langle m_W,\xi\rangle = Y_\xi \lrcorner d \sigma_W.
$$

\subsection{Definition of Balanced Metrics}

We next apply the general discussion in the previous section to calculate the moment map for the $U_N$ invariant form $\Omega_{\alpha\beta}$ on $Z$.  Given $u \in \Map$ and $x \in X$, we regard $u(x) \in \mathbb{G}$ as a surjective map
$$
u(x)\colon \CC^N \to \CC^r \cong \CC^N/\Ker \; u(x) = \mathcal{U}|_{u(x)}
$$
(here we abuse notation in that strictly speaking $u(x)$ is only an isomorphism class of such quotients, but this will not matter in the sequel).

We denote by $\mu_{FS}$ the moment map on $(\mathbb{G},\omega_{FS})$ for the $U_N$-action, and $\mu_{FS}^0$ the trace free part of $\mu_{FS}$ associated to the $SU_N$ action. 

\begin{proposition}\label{prop:mmap}
The map $\mu_{\alpha,\beta}\colon Z\to \mathfrak{u}_N$
given by
\begin{align}\label{eq:mualphabetau}
\mu_{\alpha,\beta}(u,\phi)=& \int_X u^* \mu_{FS} \,\omega^{[n]}  \\
&+\frac{\beta}{4} \int_X \frac{2 \sqrt{-1} \alpha }{1+ \alpha|\phi|_{u^*h_{FS}}^2} u(x)^*(u(x)u(x)^*)^{-1}\Lambda  [\phi_x,\phi_x^*]u(x) \omega^{[n]}\nonumber
\end{align}
is a moment map for $\Omega_{\alpha,\beta}$, where $\phi\in Z_u = H^0(\End(u^*\mathcal U)\otimes \Omega_X^1)$ and $\phi_x$ denotes evaluation at $x \in X$.
\end{proposition}
\begin{proof}
We apply Lemma \ref{lem:mmap0} with $(W,\gamma)$ replaced with $(W_x,\Gamma_x)$ from \eqref{def:Wx} and \eqref{def:Gammax}.  Then differentiating under the integral sign, as in the proof of Proposition \ref{prop:kahlerform}, the moment map we desire is
\begin{align*}
\langle \mu_{\alpha\beta}(u,\phi),\xi \rangle =& \int_X u^* \langle \mu_{FS},\xi \rangle \omega^{[n]}  \\
&+\frac{\beta}{4} \int_X \frac{2 \sqrt{-1} \alpha }{1+ \alpha|\phi|_{u^*h_{FS}}^2}\operatorname{Im} (A Y_\xi ,1_Z)_{h_{W_x}}(\phi_x) \omega^{[n]}.
\end{align*}
We are thus left to calculate the function 
$$
\operatorname{Im}(A Y_\xi ,1_Z)_{h_{W_x}} \colon W_x \to \RR
$$
for a given $x \in X$. By definition, $W_{x|u} = \End\; \mathcal{U}_{u(x)} \otimes \Omega^1_{X}|_{x}$ and $\phi_x$ denotes the evaluation of $\phi$ at $x$. 

We let $S:= u(x) \in \mathbb{G}$. %and regard it as a surjective map
%$$
%S \colon \CC^N \to \CC^r \cong \CC^N/\Ker \; C = \mathcal{U}_{u(x)}.
%$$
The induced $U_N$ action on the universal bundle $\mathcal U$ is by definition as follows: for $v\in \mathcal U_S$ and $g\in U_N$ take the pseudoinverse $S^* (SS^*)^{-1}v$ to obtain the point in $\mathbb C^N$ representing the $v$ that is orthogonal to $\ker S$, then act with $g$ and then take the image with $S$, i.e. $v\mapsto Sgv S^* (SS^*)^{-1}v$.    Thus for $g\in U_N$ the action on $\phi_x$ is
$$
\phi_x \to S g S^*(SS^*)^{-1}\phi_x S g^{-1} S^*(SS^*)^{-1}.
$$
Taking appropriate holomorphic coordinates on $W_x$ (so that the Chern connection of $h_W$ vanishes at $u$), one calculates that
$$
A Y_\xi = [S \xi S^*(SS^*)^{-1},\phi_x],
$$
and from this it follows that
\begin{align*}
\operatorname{Im}(A Y_\xi ,1_Z)_{h_{W_x}}(\phi_x) & = \operatorname{Im} \Lambda \tr [S \xi S^*(SS^*)^{-1},\phi_x] \phi_x^*, \\
& = \tr  \xi S^*(SS^*)^{-1}\Lambda [\phi_x,\phi_x^*] S,
\end{align*}
which completes the proof.
\end{proof}

We now apply the above formula when $u$ is the map $u_{\underline{s}}$ induced by a basis $\underline{s}$ of $H^0(E \otimes L^k)$, and $\phi$ is induced by a certain fixed Higgs field.  Then  the moment map becomes a function of the basis $\underline{s}$.  

\begin{definition}
For any $h\in \Met(E)$ we let
\begin{equation}
\mathfrak{C}_{\alpha\beta}(h):= \frac{\alpha \beta}{1+\alpha\vert \phi\vert^2_{h}}\Lambda [\phi,\phi^{*{h}}].\label{frakC}
\end{equation}
\end{definition}

We set
$$h_{\underline{s}}:=u_{\underline{s}}^*{h_{FS}} \otimes h_L^{-k} \in \Met(E)$$
so, by definition
\begin{equation}\label{eq:fssi}
\sum_i s_i \otimes s_i^{*h_{\underline{s}}\otimes h_L^k} = \frac{N}{\rk(E) V}\Id_E.
\end{equation}
Then Proposition  \ref{prop:mmap} gives
\begin{align}\label{eq:mubasis}
(\mu_{\alpha\beta}(\underline{s}))_{ij} 
=&\frac{- \sqrt{-1}}{2}\int_X  (s_i,s_j)_{h_{\underline{s}}} \omega^{[n]} \\
&+ \frac{\beta}{4} \int_X \frac{2\alpha \sqrt{-1}}{1+ \alpha|\phi|_{h_{\underline{s}}}^2} (s_i,\Lambda [\phi_x,\phi_x^{*h_{\underline{s}}}]s_j)_{h_{\underline{s}}} \omega^{[n]},\nonumber\\
=&\frac{- \sqrt{-1}}{2}\int_X  (s_i,s_j)_{h_{\underline{s}}} \omega^{[n]} +\frac{ \sqrt{-1}}{2} \int_X  (s_i,\mathfrak{C}_{\alpha\beta}(h_{\underline{s}}) s_j)_{h_{\underline{s}}} \omega^{[n]}.
\end{align}
Taking the trace of the matrix \eqref{eq:mubasis} 
\begin{align*}
\tr (\mu_{\alpha\beta}(\underline{s})) &= \frac{- \sqrt{-1}}{2} \int_X \sum_i \tr(s_i \otimes s_i^{*h_{\underline{s}}\otimes h_L^k}) \omega^{[n]}
\\&\quad + \frac{ \sqrt{-1}}{2} \sum_i \int_X \tr(s_i \otimes s_i^{*h_{\underline{s}}\otimes h_L^k} \mathfrak{C}_{\alpha\beta}(h_{\underline{s}})) \omega^{[n]},\\
&=\frac{- \sqrt{-1}N}{2}, 
\end{align*}
where we have used \eqref{eq:fssi} and the fact that $\mathfrak{C}_{\alpha\beta}$ is a trace-free endomorphisms of $E$.  Thus the induced $SU_N$ action has moment map \begin{equation}\label{momentmapsuN}\mu_{\alpha\beta}^0:=\mu_{\alpha,\beta}+ \frac{ \sqrt{-1}}{2}\Id_N\in \mathfrak s\mathfrak u_N.                                                                                                                                                                                                                                                                                                                                                  \end{equation}

%\begin{lemma}
%A basis $\underline{s}$ for $H^0(E\otimes L^k)$ is a zero of the moment map 
%$$\mu_{\alpha,\beta}^0\colon Z \to \mathfrak s\mathfrak u_N$$  if and only if
%\end{lemma}

\begin{definition}
Let $k\in \mathbb N$ and $\alpha,\beta\in \mathbb R_{>0}$ with $\beta<\frac{1}{2(r-1)}$.  We say $h\in \Met(E)$ is \emph{balanced} with respect to $(\alpha,\beta,k)$ if there is a basis $\underline{s}$ for $H^0(E\otimes L^k)$ such that $h=h_{\underline{s}}\otimes h_L^{-k}$ and $\mu_{\alpha,\beta}^0(\underline{s})=0$, i.e. if and only if
$$\int_X (s_i,(\Id_E-\mathfrak{C}_{\alpha\beta}(h_{\underline{s}})s_j)_{h_{\underline{s}}}\omega^{[n]}=\delta_{ij}$$
When this holds we refer to $\underline{s}$ as a \emph{balanced basis} and to the inner product on $H^0(E\otimes L^k)$ that makes $\underline{s}$ orthonormal as a \emph{balanced metric} on $H^0(E\otimes L^k)$. 
\end{definition}

\subsection{Balanced condition and the Bergman function}

%Recall $\Met(E)$ denotes the space of hermitian metrics on $E$ and $\mathfrak{B}_{k}$ be the space of hermitian inner products on $V:=H^{0}(X,E\otimes L^k)$. Following Donaldson, we connect these as follows:
%\begin{definition}
% Define $$\Hilb:  \Met(E)\rightarrow \mathfrak{B}_{k}$$ by
% $$\langle s,t \rangle_{\Hilb(h)}=  \int_{X}
%\langle s(x),t(x) \rangle_{h \otimes h_L^k }\omega^{[n]}(x),$$
%for any $s,t \in H^{0}(X,E\otimes L^k)$.   In other words, $\Hilb(h)$ is the induced $L^2$-inner product on $H^0(E\otimes L^k)$.  
%\end{definition}

%\todo[inline]{do we actually need the following definition? -- yes but later}
%\begin{definition}
%Define $$\FS:\mathfrak{B}_{k} \rightarrow \Met(E)$$ as follows.  For an inner product $H$ in $\mathfrak{B}_{k}$, $\FS(H)$ is the unique metric
%on $E$ such that  \begin{equation}\label{FSmap}\sum s_{i}\otimes s_{i}^{*_{\FS(H) \otimes h_L^k}}=\frac{N}{\rk(E) \Vol(X)}\Id_E,\end{equation} where
%$s_{1},...,s_{N}$ is an orthonormal basis for $V$ with
%respect to $H$ and $N=\dim V$.   That is $\FS(H)$ is the induced Fubini-Study metric on $E$ induced by $H$.
%\end{definition}

\begin{definition}(Bergman Function)
Given $h\in \Met(E)$ the Bergman function $B_k(h)$ of $h$ is the restriction to the diagonal of $X\times X$ of the kernel of the $L^2$-projection from $L^2(X,E\otimes L^k)$ to $H^0(E\otimes L^k)$.  That is,
$$B_k(h)=\sum_{i=1}^{N} s_i\otimes s_i^{*h\otimes h_{L}^k}$$
where $\{s_i\}$ are a basis for $H^0(E\otimes L^k)$ that is orthonormal with respect to the $L^2$-inner product $L^2(h)$.  
\end{definition}

For later use we record the following special case of the well-known asymptotics of the Bergman function \cite{BBS,C1,F1,MM1,T1,Y1,Z1}.

\begin{theorem}(Asymptotic Expansion of the Bergman Function)\label{thm:asymptoticbergman}
Let $h\in \Met(E)$.  For any $p$ and any $r$ is a $C^{r}$-asymptotic expansion of the Bergman function
$$ B_k(h) = a_0 k^n + a_1 k^{n-1} + \cdots a_p k^{n-p} + O(k^{n-p-1})$$
where $a_j\in C^{\infty}(\End(E))$ are universal coefficients that depend on the curvature of $h$. Moreover the $O(k^{n-p-1})$ remainder term can be taken uniformly as $h$ ranges in a compact set of $\Met(E)$. The top two coefficients are given by
\begin{equation}
B_k(h) = \Id_E k^n + \left( \sqrt{-1}\Lambda F_h + \frac{\Scal(\omega)}{2} \Id_E\right) k^{n-1} + O(k^{n-2}).\label{eq:asymptoticbergmantwo}
\end{equation}
 \end{theorem}

Observe next that for any $h\in \Met(E)$ the norm of the eigenvalues of the operator $\mathfrak C_h=\frac{\alpha\beta}{1+ \alpha |\phi|_h^2} \Lambda [\phi,\phi^{*h}]$ from are bounded by $2\beta$, by \eqref{eq:ineqbracket}.  So from now on we assume $\beta<\frac{1}{2}$ so the hermitian operator $\Id_E - \mathfrak C_{\alpha\beta}(h)$ is strictly positive.

\begin{proposition}(Balanced Metrics in terms of the Bergman Function)
A metric $h\in \Met(E)$ is balanced with respect to $(\alpha,\beta,k)$ if and only if 
$$ B_k( \widehat{h}) = \frac{N}{\rk(E) V} (\Id_E - \mathfrak{C}_{\alpha\beta}(h))$$
where
\begin{equation}
  \widehat{h} := h(\Id_E-\mathfrak{C}_{\alpha\beta}(h)).\label{def:h'}
\end{equation}
\end{proposition}
\begin{proof}
Suppose $h$ is balanced, so by definition $h= h_{\underline{s}}$ for some $\underline{s}\in \mathfrak{B}_k$ such that
\begin{equation}
\int_X (s_i, (\Id_E - \mathfrak{C}_{\alpha\beta}(h)) s_j)_{h} \omega^{[n]} = \delta_{ij}.\label{eq:proofbergmanl2}
\end{equation} 
By the definition of $ \widehat{h}$ in the statement of the proposition this implies $\{s_i\}$ are orthonormal with respect to the $L^2$-metric induced by $ \widehat{h}$.  Hence
\begin{equation}
B( \widehat{h}) = \sum_i s_i \otimes s_i^{*{ \widehat{h}\otimes h_L^k}} = \left(\sum_i s_i \otimes s_i^{*{h\otimes h_L^k}} \right) {h^{-1}}{ \widehat{h}} = \frac{N}{\rk(E)V} (\Id_E - \mathfrak{C}_{\alpha\beta}(h))\label{eq:Bh'}
\end{equation}
where we have used $\sum_i s_i\otimes s_i^{*h\otimes h_L^k}=\frac{N}{\rk(E) V}\Id_E$ since $h=h_{\underline{s}}$ is obtained via the Fubini-Study metric associated to $\underline{s}$.  

Conversely suppose $B_k( \widehat{h}) = \frac{N}{\rk(E) V} (\Id_E - \mathfrak{C}_{\alpha\beta}(h))$ where $ \widehat{h}$ is as in \eqref{def:h'} and let $\underline{s}$ be a basis for $H^0(E\otimes L^k)$ that is orthonormal with respect to the $L^2$-norm induced by $ \widehat{h}$.  Then \eqref{eq:proofbergmanl2} holds definition of $ \widehat{h}$ and the same calculation as in \eqref{eq:Bh'} gives $\sum_i s_i \otimes s_i^{*h\otimes h_L^k} = \frac{N}{\rk(E) V}$.  Hence $h = h_{\underline{s}}$ and so $h$ is balanced.
\end{proof}

\section{Balanced Metrics and the Hitchin Equation}

So far this discussion has allowed the most general values of $\alpha,\beta$.  From now on we specialise and set
\begin{equation} \alpha := \frac{2(\rk(E)-1)}{k} \text{ and }\beta:=\frac{1}{2(\rk(E)-1)} \label{fix}\end{equation}
where $k\in \mathbb N$.  We say $h\in \Met(E)$ is \emph{balanced at level $k$} if it is balanced with respect to $(\alpha,\beta,k)$ for this choice of $\alpha,\beta$.  For convenience we repeat this definition in full:

\begin{definition}
We say a hermitian metric $h\in \Met(E)$ is \emph{balanced at level $k\in \mathbb N$} if
$$B_k( \widehat{h}) = \frac{N}{\rk(E) V} (\Id_E - \mathfrak{C}_k(h)),$$
where
\begin{equation}  \widehat{h} = h(\Id_E - \mathfrak{C}_k(h)),\label{eq:hat}\end{equation}
and
$$\mathfrak{C}_k(h):= \frac{1}{k+2(\rk(E)-1)\vert \phi\vert^2_{h}}\Lambda [\phi,\phi^{*{h}}].\label{frakCrepeat}$$
\end{definition}

\begin{lemma}\label{lem:reformulationbalanced}
An $h\in \Met(E)$ is balanced at level $k$ if and only if
 $$\sum_{i=1}^N s_i\otimes s_i^{*h\otimes h_L^k} = \frac{N}{\rk(E)V}\Id_E,$$
where the $s_i\in H^0(X,E\otimes L^k)$ form a holomorphic basis that is orthonormal with respect to $L^2(\hat{h})$.
\end{lemma}
\begin{proof}
If $s_i$ are as in the statement then
$$B_k(\widehat{h}) = \sum_i s_i \otimes s_i^{*\widehat{h}\otimes h_L^k} = \sum_i s_i \otimes s_i^{*h\otimes h_L^k} (\Id_E - \mathfrak C_k(h))$$
and so the lemma is just a reformulation of the definition.
\end{proof}

\begin{remark}
As mentioned in the introduction, the precise value of $\beta$ is not important.  We want $\beta<1/2$ to ensure $\Id-\mathfrak{C}_k(h)$ is invertible for any $h\in \Met(E)$, and later we will want to apply our positivity result Lemma \ref{prop:kahlerform}, for which it is sufficient to take $\beta$ to be anything smaller than $2(\rk(E)-1))^{-1}$.  To ensure the balanced condition is related to the Hitchin equation we must take $\alpha = O(1/k)$.  Since we have arranged $\alpha\beta=1/k$ we will see presently the balanced condition is related to $\Lambda (\sqrt{-1}F_h +  [\phi,\phi^*]) = \frac{\mu(E)}{V}\Id_E$.    Had we made a different choice, say $\alpha\beta = \tau/k$ for some $\tau \in \mathbb R_{>0}$, then the balanced condition would be related instead to the equation $\Lambda (\sqrt{-1}F_h + \tau  [\phi,\phi^*]) = \frac{\mu(E)}{V}\Id_E$.
\end{remark}

%\subsection{Identifying the limit of balanced metrics}

\begin{theorem}\label{thm1}
Suppose $h_k$ is a sequence of hermitian metrics on $E$ such that $h_k$ is balanced at level $k$ that converge to $h_{\infty}$ as $k$ tends to infinity.  Then $h_{\infty}$ satisfies the equation
\begin{equation}
 \Lambda \left( \sqrt{-1} F_{h_\infty} + [\phi, \phi^{*h_\infty}]\right) = \left(\frac{\mu(E)}{V} - \frac{1}{2}\Scal_0(\omega)\right)\Id_E. \label{eq:f2}
 \end{equation}
 Thus, after a conformal change, $h_{\infty}$ satisfies the Hitchin equation.
\end{theorem}
\begin{proof}
We have
\begin{equation}
\mathfrak{C}_k(h_k) = \frac{1}{k} \Lambda [\phi,\phi^{*h_k}] + O\left(\frac{1}{k^2}\right).\label{eq:expansionmathfrakC}
\end{equation} 

This in particular implies $\mathfrak{C}_k(h_k)\to 0$ as $k\to \infty$, and so $ \widehat{h_k}:= h(\Id_E-\mathfrak{C}_k(h_k))$ tends to $h_{\infty}$ as $k\to \infty$.      The balanced hypothesis says 
\begin{equation}
B_k( \widehat{h_k}) = \frac{N}{\rk(E) V} (\Id- \mathfrak{C}_k(h_k)). \label{eq:balancedhypothesis}
\end{equation} 
Certainly the $ \widehat{h_k}$ lie in a bounded set, so a diagonal argument with the asymptotic of the Bergman function \eqref{eq:asymptoticbergmantwo} yields
\begin{equation}
B_k( \widehat{h_k}) = \Id_E k^n + \left(  \sqrt{-1} \Lambda F_{ \widehat{h_k}} + \frac{\Scal(\omega)}{2} \Id_E\right) k^{n-1} + O(k^{n-2}).\label{eq:bergmandiagonal}
\end{equation} 
%\begin{equation}
%(\Id-\mathfrak C_k(h_k)) B_k( \widehat{h_k}) = \Id k^n + \left( i\Lambda F_{ \widehat{h_k}}  - [\phi,\phi^{*h_k}] + \frac{\Scal(\omega)}{2} \Id_E\right) k^{n-1} + %O(k^{n-2})
%\end{equation} 
On the other hand, by the Riemann-Roch theorem
\begin{equation}
\frac{N}{\rk(E) V}\Id_E = \Id_E  k^n + \left(\frac{\mu(E)}{V} - \frac{\overline{S}}{2}\right)\Id_Ek^{n-1} + O(k^{n-2})\label{eq:rroch}
\end{equation} 
where, we recall, $\overline{S}$ is the average of $\Scal(\omega)$.   So putting \eqref{eq:expansionmathfrakC} through \eqref{eq:rroch} together gives 
$$ \sqrt{-1} \Lambda F_{ \widehat{h_k}} +\Lambda  [\phi,\phi^{*h_k}] + \frac{\Scal(\omega)}{2} \Id_E = \left(\frac{\mu(E)}{V} - \frac{\overline{S}}{2}\right)\Id_E + O\left(\frac{1}{k}\right),$$
and taking $k$ to infinity proves the first statement.

The final statement about the conformal change is standard, for if $h: = e^u h_{\infty}$ for some real function $u$ then $[\phi,\phi^{*h}] = [\phi,\phi^{h_{\infty}}]$ and  $\sqrt{-1}\Lambda F_{h} = \sqrt{-1}\Lambda F_{h_{\infty}} + \Delta u \Id_E$.  Thus if one chooses $u$ such that $\Delta u =   \frac{1}{2} \Scal_0(\omega)$ (which is possible by Hodge-Theory as $\Scal_0(\omega)$ has average $0$) then $h$ satisfies the Hitchin equation.
%\julius{Added Jan 5}
\end{proof}

\section{Balanced Metrics and Gieseker Stability}

\begin{theorem}\label{thm3}
There exists a $k_0$ such that for all $k\ge k_0$ the following holds:  if $(E,\phi)$ admits an $h\in \Met(E)$ that is balanced at level $k$ then it is Gieseker semistable.  In particular, if $(E,\phi)$ is irreducible then it is Gieseker stable.
\end{theorem}

\begin{proof}
For the proof we follow the lines of \cite[Theorem 4.2]{GFR}.  Assume $(E,\phi)$ is balanced for $k\gg 0$ and let $\{s_{bal}\}$ be a balanced basis at level $k$.     Decompose the moment map from \eqref{momentmapsuN} and Proposition \ref{prop:mmap} as
\begin{equation}\mu^0_{\alpha,\beta}=\int_X u^*\mu_{FS}^0\omega^{[n]}+\mu^\phi,\label{eq:decompose}\end{equation}
where we continue to impose our choice of $\alpha$ and $\beta$ from \eqref{fix}.  Then by definition
$$ \mu^0_{\alpha,\beta}(s_{bal}) =0.$$

 Let $F$ be a saturated coherent subsheaf (so $E/F$ is torsion-free) and set $V'=H^0(F\otimes L^k)$. Consider the one-parameter subgroup 
$$\lambda: \mathbb{C}^*\rightarrow SL(N,\mathbb{C})$$
with $\lambda(t)=t$ on $V'$ and $\lambda(t)=t^{-\nu}$ on ${V'}^{\perp}$, where the orthogonal complement here is taken with respect to the inner product on $H^0(E\otimes L^k)$ that makes the basis $\{s_{bal}\}$ orthonormal.  To ensure this is an $SL(N)$ action, $\nu$ is taken to be $$\nu=\frac{h^0(F\otimes L^k)}{h^0(E\otimes L^k)-h^0 (F\otimes L^k)},$$ so the generator of the action is given by $$\xi=\sqrt{-1}\begin{pmatrix}                                                                                     
  \Id_{V'} & 0 \\
   0 & -\nu \Id_{{V'}^\perp}
   \end{pmatrix}.$$
Using the $\mathbb{C}^*$ action, we obtain an equivariant family of coherent sheaves with general fibre isomorphic to $(E(k),\{s_{bal}\})$ and central fibre isomorphic to $(F\otimes  L^k \oplus (E/F)\otimes L^k, \{s'\}\oplus \{s''\})$ where $\{s'\}$ a basis of $H^0(F\otimes L^k)$ and $\{s''\}$ a basis of $H^0((E/F) \otimes L^k)$.   So using the balanced hypothesis,
\begin{equation}\label{eq:weightineq}
w(\underline{s},\lambda):= \lim_{t \to +\infty} \langle \mu^0_{\alpha,\beta}(e^{it\zeta}u_{\underline s}),\zeta\rangle = \int_0^\infty |Y_{\zeta|e^{it\zeta}u_{\underline s}}|^2dt \geq 0,
\end{equation}
where $t\in\RR$ and $Y_{\zeta|e^{i\zeta} u_{\underline s}}$ denotes the infinitesimal action of $\zeta$ on $u_{\underline s}\in \Map$. Moreover equality holds only if $i\zeta$ is an infinitesimal automorphism of $u_{\underline s}$ and hence is excluded when $(E,\phi)$ simple. 

Now, the decomposition \eqref{eq:decompose} gives a decomposition $w(\underline{s},\lambda) = w^{FS} + w^\phi$.  The computation of the first term $w^{FS}$ for (i.e. the calculation without the Higgs field) is performed in \cite{GFR}, and is given by 
\begin{align*}
w^{FS}:=&\int_X u_{\underline{s}}^*\mu_{FS}^0\omega^{[n]},\\
=& \left(\frac{h^0(E\otimes L^k)}{\rk(E)}- \frac{h^0(F\otimes L^k)}{\rk(F)}\right)\frac{V\rk(E)\rk(F)}{h^0(E\otimes L^k)-h^0(F\otimes L^k)}. 
\end{align*}
Next we compute the weight $w^\phi$ for $\mu^{\phi}(s)$,  i.e by definition
$$w^\phi(s,\lambda)=\lim_{t\rightarrow + \infty} \langle \mu^\phi(e^{it\xi} s),\xi\rangle.$$
In order to do so, for any bundle $G$ and basis $\underline{b}$ for $H^0(G\otimes L^k)$,  let $h_{FS}(\underline{b})$ denote the Fubini-Study metric induced by $\underline{b}$, and denote by $\pi_F$ the orthogonal projection  $E\to F$ taken with respect to $h_{FS}(\underline{s})$.  Then, using that $\mathfrak{C}_k = \mathfrak{C}_k(h_{FS}(\underline{s}))$ is trace-free,
\begin{align*}
w^\phi(s,\lambda)=&-\sqrt{-1}\tr_{V'}\left(\left(\int_X \langle s'_j,\mathfrak{C}_k s'_l\rangle_{h_{FS}(s')}\omega^{[n]}\right)_{j,l}\xi_{V'}\right)\\
&-\sqrt{-1}\tr_{{V'}^\perp}\left(\left(\int_X \langle s''_j,\mathfrak{C}_ks''_l\rangle_{h_{FS}(s'')}\omega^{[n]}\right)_{j,l}\xi_{{V'}^\perp}\right) ,\\ 
=&-\left(\int_X \tr(\pi_F \mathfrak{C}_k)\omega^{[n]} - \nu  \int_X \tr(\pi_{F^\perp} \mathfrak{C}_k)\omega^{[n]}\right) ,\\
=&-\left((1+\nu)\int_X \tr(\pi_F \mathfrak{C}_k)\omega^{[n]} - \nu  \int_X \tr( \mathfrak{C}_k)\omega^{[n]}\right),\\
=&-(1+\nu)\int_X \tr(\pi_F \mathfrak{C}_k)\omega^{[n]}.
\end{align*}
Now, we have 
\begin{align*}
 \int_X \tr(\pi_F \mathfrak{C}_k)\omega^{[n]} =&\int_X \frac{\alpha\beta}{1+\alpha\vert \phi\vert^2} \tr(\pi_F \Lambda \ [\phi, \phi^*])\omega^{[n]}.
\end{align*}
A computation shows that $\tr(\pi_F \Lambda  [\phi, \phi^*])= \vert \pi_F \phi(Id - \pi_F)\vert^2$, see \cite[Proposition 2.16]{Went}. Hence, we have 
$$\int_X \tr(\pi_F \mathfrak{C}_k)\omega^{[n]}= \Big\Vert \frac{\alpha\beta}{\sqrt{1+\alpha\vert \phi\vert^2}} \pi_F \phi(Id - \pi_F)\Big\Vert_{L^2}^2,$$
and so in total
$$w^\phi(s,\lambda) = -(1+\nu)\Big\Vert \frac{\alpha\beta}{\sqrt{1+\alpha\vert \phi\vert^2}} \pi_F \phi(Id - \pi_F)\Big\Vert^2_{L^2}.$$
We observe that this term vanishes if and only if the Higgs field splits.

%Now, by Kempf-Ness theory, the existence of a balanced metric implies that $w^{FS}+w^\phi\geq 0$ which from the above is equivalent to 
Therefore with \eqref{eq:weightineq}, we obtain
$$\frac{h^0(E\otimes L^k)}{\rk(E)}- \frac{h^0(F\otimes L^k)}{\rk(F)}- \frac{h^0(E\otimes L^k)}{V\rk(E)\rk(F)}  \Big\Vert \frac{\alpha\beta}{\sqrt{1+\alpha\vert \phi\vert^2}} \pi_F \phi(Id - \pi_F)\Big\Vert_{L^2}^2 \geq 0$$
for all $k$ sufficiently large (with strict inequality when $(E,\phi)$ is simple) which proves the theorem. 
\end{proof}

 \vspace{1cm}
 

\begin{thebibliography}\frenchspacing\smallbreak

\bibitem{AC1}
    L. \'Alvarez-C\'onsul,
    \emph{Some results on the Moduli spaces of quiver bundles}, Geom. Dedicata {\bf 139} (2009) 99--120.

\bibitem{AC2}
    L. \'Alvarez-C\'onsul and O. Garc\'ia-Prada,
    \emph{Kobayashi-Hitchin correspondence, quivers and vortices}, Comm. Math. Phys. \textbf{238} (2003) 1-33.

\bibitem{AC3}
    L. \'Alvarez-C\'onsul and O. Garc\'ia-Prada,
    \emph{Dimensional reduction and quiver bundles}, J. Reine Angew. Math. \textbf{556} (2003) 1-46.

\bibitem{Araujo}
    A. de Araujo,
    \emph{Generalized Quivers, Orthogonal and Symplectic Representations, and Hitchin--Kobayashi Correspondences},  ArXiv:math/1508.00460 (2015).

\bibitem{BBS} R. Berman, B. Berndtsson \and J. Sj\"ostrand \emph{A direct approach to Bergman Kernel assymptotics for positive line bundles}, Ark.\ Mat.\ \textbf{46} no.\ 2  (2007), 197--217.

%\bibitem{ACK}
%    L. \'Alvarez-C\'onsul and A. King,
%    \emph{A functorial construction of moduli of sheaves}, Invent. Math. \textbf{168} (2007) 613-666.

\bibitem{C1} David Catlin,
        \emph{The Bergman kernel and a theorem of Tian}, Analysis and geometry in several complex variables (Katata, 1997), Trends Math., Birkh\"auser Boston, Boston, MA, (1999), 1--23.

\bibitem{D1}
        S. K. Donaldson,
        \emph{Scalar curvature and projective embeddings, I},
        J. Diff. Geom. \textbf{59} (2001) 479--522.

\bibitem{D2}
		\bysame, \emph{Some numerical results in complex differential geometry},
        Pure Appl. Math. Q. \textbf{5} (2009) 571--618.
        
\bibitem{DonWi}
    R. Donagui and M. Wijnholt,
    \emph{Gluing Branes, I}, ArXiv:1104.2610  (2011).
    
    \bibitem{Douglas}
M. Douglas, L. Robert, S. Lukic, R. Reinbacher \emph{Numerical solution to the Hermitian Yang-Mills equation on the Fermat quintic}
 J. High Energy Phys., \textbf{12}, 083, (2007). 
 
\bibitem{F1}
        C. Fefferman, \emph{The {B}ergman kernel and biholomorphic mappings of pseudoconvex domains}, Invent. Math. \textbf{26}, (1974), 1--65.        

        \bibitem{GFR}
        M. Garcia-Fernandez and J. Ross, \emph{Balanced metrics on twisted Higgs Bundles},  ArXiv:1401.7108 (2014).
        
        \bibitem{Griffiths} P. Griffiths and J. Harris \emph{Principles of algebraic geometry}, Reprint of the 1978 original. Wiley Classics Library. John Wiley \& Sons, Inc., New York, (1994). 

\bibitem{Gomez1}
T. Gomez and I. Sols \emph{Stable Higgs G-sheaves}, Rev. Mat. Iberoam. \textbf{24} (2008), no. 2, 703--719.

\bibitem{Gomez2} T. Gomez and I. Sols \emph{The Hermite-Einstein equation and stable principal bundles (an updated survey)} Geom. Dedicata \textbf{139} (2009), 83--98. 
        
        \bibitem{Hitchin} N. Hitchin \emph{The self-duality equations on a Riemann surface} 
Proc. London Math. Soc. (3) \textbf{55} (1987), no. 1, 59--126.

\bibitem{Huybrechts}
	D. Huybrechts, \emph{Complex Geometry} Universitext,	Universitext. Springer-Verlag, Berlin, 2005.
\bibitem{Keller}
		J. Keller, \emph{Canonical metrics for Vortex type equations}, Math. Annalen, \textbf{337}, 923--979 (2007).
		
\bibitem{KellerLukic}
		J. Keller and S. Lukic, \emph{Numerical Weyl-Petersson metrics on moduli spaces of Calabi-Yau manifolds}, Journal of Geometry and Physics, \textbf{92}, (2015), 252--270 .

		
\bibitem{Ki2}, F. Kirwan, \emph{Cohomology of quotients in symplectic and algebraic geometry},   {Princeton University Press}, {(1984)}
 
		
		
		
		
\bibitem{MM1}
        X. Ma and G. Marinescu,
        \emph{Holomorphic Morse inequalities and Bergman kernels}, Progress in Mathematics, vol. 254, Birkh\"auser Verlag, Basel, (2007).
		
		\bibitem{McDuff} D. McDuff and D. Salamon  \emph{Introduction to symplectic topology}, Second edition. Oxford Mathematical Monographs. The Clarendon Press, Oxford University Press, New York, (1998). x+486 pp. ISBN: 0-19-850451-9 

		\bibitem{Rezza} Z. Lu and  R. Seyyedali 	 \emph{Extremal metrics on ruled manifolds},  Adv. Math. \textbf{258} (2014), 127--153. 
\bibitem{Okonek}
C. Okonek and A.  Teleman \emph{The coupled Seiberg-Witten equations, vortices, and moduli spaces of stable pairs},  Internat. J. Math. \textbf{6} (1995), no. 6, 893--910. 


\bibitem{Schmitt1}
A. Schmitt,
\emph{A universal construction for moduli spaces of decorated vector bundles over curves}, ArXiv:math/0006029v3 (2004)

\bibitem{Schmitt2}
		\bysame,
\emph{Moduli for decorated tuples of sheaves and representation spaces for quivers}, Proc. Indian Acad. Sci. (Math. Sci.), \textbf{115}, No. 1, (2005), 15-49.

\bibitem{Schmitt3}
	\bysame,
\emph{Geometric invariant theory and decorated principal bundles}, Zurich lectures in Advanced Mathematics (2008).


\bibitem{Simpson}		
		C. Simpson, \emph{Moduli of representations of the fundamental group of a smooth projective variety}, I. Inst. Hautes \'Etudes Sci. Publ. Math. {\bf 79} (1994) 47--129.

\bibitem{SimpsonII}		
		C. Simpson, \emph{Moduli of representations of the fundamental group of a smooth projective variety II} \emph{Moduli of representations of the fundamental group of a smooth projective variety. II.}
Inst. Hautes Études Sci. Publ. Math. No. {\bf 80} (1994), 5--79 (1995). 		
		
\bibitem{Simpson1}
		C. Simpson, \emph{Constructing variations of Hodge Structure using Yang--Mills Theory and application to uniformization}, J. Amer. Math. Soc. {\bf 1} (1988) 867--918.

\bibitem{T1}
  	G. Tian, \emph{On a set of polarized {K}\"ahler metrics on algebraic manifolds}, J. Differential Geom. \textbf{32} no.~1, (1990), 99--130.
  	
\bibitem{UY}
        K. K. Uhlenbeck and S.-T. Yau,
        \emph{On the existence of Hermitian--Yang--Mills connections on stable bundles over compact K\"ahler manifolds}, Comm. Pure and Appl. Math. \textbf{39-S}, 1986, 257--293; \textbf{42}, (1989), 703--707.

\bibitem{WangLi}
		L. Wang, \emph{Bergman Kernel and stability of holomorphic vector bundles with sections}, MIT Thesis, (2003).

\bibitem{W1}
        X. Wang, \emph{Canonical metric and stability of vector bundles over a projective manifold},
        Ph.D. Thesis, Brandeis University, (2002).

\bibitem{W2}
        \bysame, \emph{Balance point and stability of vector bundles over a projective manifold},
        Math. Res. Lett. \textbf{9} no. 2-3, (2002), 393--411.
        ,
\bibitem{Went} R. Wentworth \emph{ Higgs bundles and local systems on Riemann surfaces}, To appear in CRM Advanced Courses in Mathematics.  ArXiv:1402.4203, (2014). 

\bibitem{Y1} S.-T. Yau,
        \emph{Nonlinear analysis in geometry}, Monographies de L'Enseignement  Math\'ematique \textbf{33}, (1986).

\bibitem{Z1} S. Zelditch,
        \emph{Szeg\"{o} kernels and a theorem of {T}ian}, Internat.  Math. Res. Notices, \textbf{6}, (1998), 317--331.

\end{thebibliography}
\end{document}